\documentclass[12pt,reqno]{amsart}
\usepackage{graphicx}
\usepackage{psfrag}
\usepackage{pxfonts}
\usepackage{mathrsfs}
\usepackage{color}
\usepackage{amsmath,amssymb}

\numberwithin{equation}{section}

\newcommand{\x} {\tilde{x}}

\newtheorem{theorem}{Theorem}[section]

\newtheorem{corollary}[theorem]{Corollary}
\newtheorem{proposition}[theorem]{Proposition}
\newtheorem{lemma}[theorem]{Lemma}
\newtheorem{definition}[theorem]{Definition}

\theoremstyle{remark}
\newtheorem{remark}[theorem]{Remark}

\begin{document}

\thanks{}

\author{F. Micena}
\address{Departamento de Matem\'atica,
  IM-UFAL Macei\'{o}-AL, Brazil.}
\email{fpmicena@gmail.com}

\author{A. Tahzibi}
\address{Departamento de Matem\'atica,
  ICMC-USP S\~{a}o Carlos-SP, Brazil.}
\email{tahzibi@icmc.usp.br}

\renewcommand{\subjclassname}{\textup{2000} Mathematics Subject Classification}

\date{\today}

\setcounter{tocdepth}{2}

\title{On the Unstable Directions and Lyapunov Exponents of  Anosov Endomorphisms}
\maketitle
\begin{abstract}
Despite the invertible setting, Anosov endomorphisms may have infinitely many unstable directions. Here we  prove, under transitivity assumption, that an Anosov endomorphism on a closed manifold $M,$ is either special (that is, every $x \in M$ has only one unstable direction) or for a typical  point in $M$ there are infinitely many unstable directions. Other result of this work is the semi rigidity of the unstable Lyapunov exponent of a $C^{1+\alpha}$ codimension one Anosov endomorphism and $C^1$ close to a linear endomorphism of $\mathbb{T}^n$ for $(n \geq 2).$  In the appendix we give a proof for ergodicity of  $C^{1+\alpha}, \alpha > 0,$ conservative Anosov endomorphism.
\end{abstract}

\section{Introduction}\label{section.preliminaries}

In $1970s,$ the works \cite{PRZ} and \cite{MP}  generalized the notion of Anosov diffeomorphism for non invertible maps, introducing the notion of Anosov endomorphism. We consider $M$ a $C^{\infty}$ closed manifold.

\begin{definition}\cite{PRZ} \label{defprz} Let $f: M \rightarrow M$ be a  $C^1$ local diffeomorphism. We say that $f$ is an Anosov endomorphism if there is constants $C> 0$ and $\lambda > 1,$ such that, for every $(x_n)_{n \in \mathbb{Z}}$ an $f-$orbit there is a splitting

$$T_{x_i} M = E^s_{x_i} \oplus E^u_{x_i}, \forall i \in \mathbb{Z},$$

which is preserved by $Df$ and for all $n > 0 $ we have

$$||Df^n(x_i) \cdot v|| \geq C^{-1} \lambda^n ||v||, \;\mbox{for every}\; v \in E^u_{x_i} \;\mbox{and for any} \; i \in \mathbb{Z}$$
$$||Df^n(x_i) \cdot v|| \leq C\lambda^{-n} ||v||, \;\mbox{for every}\; v \in E^s_{x_i} \;\mbox{and for any} \; i \in \mathbb{Z}.$$

\end{definition}

Anosov endomorphisms  can be defined in an equivalent way (\cite{MP}):

\begin{definition}\cite{MP} \label{defmp} A $C^1$ local diffeomorphism $f: M \rightarrow M$ is said an Anosov endomorphism if $Df$ contracts uniformly a continuous sub-bundle $E^s \subset TM$ into itself, and the action of $Df$ on $TM/E^s$ is uniformly expanding.
\end{definition}

Sakai, in \cite{SA} proved that, in fact, the definitions $\ref{defprz}$ and $\ref{defmp}$ are equivalent.

A contrast between Anosov diffeomorphisms and Anosov endomorphisms is the non-structural stability of the latter. Indeed, $C^1-$close to any linear Anosov endomorphism $A$ of Torus, Przytycki \cite{PRZ} constructed Anosov endomorphism which has infinitely many unstable direction for some orbit and consequently he showed that $A$ is not structurally stable. However, it is curious to observe that the topological entropy is locally constant among Anosov endomorphisms. Indeed, take the lift of Anosov endomorphism to the inverse limit space (see preliminaries for the definition). At the level of inverse limit space, two nearby Anosov endomorphisms are conjugate (\cite{PRZ}, \cite{BerRov}) and lifting to inverse limit space does not change the entropy.

Two endomorphisms (permitting singularities) $f_1, f_2$ are $C^1-$inverse limit conjugated, if there exists a homeomorphism $h : M^{f_1} \rightarrow M^{f_2}$ such that $h \circ \tilde{f_1} = \tilde{f_2} \circ h$ where $\tilde{f_i}$ are the lift of $f_i$ to the orbit space (see preliminaries).

Denote by $p$ the natural projection $p: \overline{M} \rightarrow M,$ where $\overline{M}$ is the universal covering. Note that an unstable direction $E^u_{\overline{f}}(y)$ projects on an unstable direction of $T_x M, x = p(y)$ following the definition $\ref{defprz}, $ that is $Dp(y) \cdot (E^u_{f}(y)) = E^u(\tilde{x}), $ where $\x =  p (\mathcal{O}(y)).$

\begin{proposition}\label{propMP}\cite{MP} $f$ is an Anosov endomorphism of $M,$ if and only if, the lift  $\overline{f}: \overline{M} \rightarrow \overline{M} $  is an Anosov diffeomorphism of $\overline{M},$ the universal cover of $M.$
\end{proposition}

 An advantage to work with the latter definition is that in $\overline{M}$ we can construct invariant foliations $\mathcal{F}^s_{\overline{f}}, \mathcal{F}^u_{\overline{f}}.$

 Given an Anosov endomorphism and $\x= (x_n)_{n \in \mathbb{Z}}$ an $f-$ orbit we denote by $ E^u(\x)$ the unstable bundle subspace of $T_{x_0}(M)$ corresponding to the orbit $(x_n)_{n \in \mathbb{Z}}.$ In  \cite{PRZ} one constructs examples of Anosov endomorphism such that $E^u({\tilde{x}}) \neq E^u (\tilde{y}),$ when $x_0 = y_0,$ but $(x_n)_n \neq (y_n)_n.$ In fact, it is possible that $x_0 \in M$ has uncountable unstable directions, see \cite{PRZ}.
 An Anosov endomorphism for which $E^u(\x)$ just depends on $x_0$ (unique unstable direction for each point) is called special Anosov endomorphism. A linear Anosov endomorphism of torus is an example of special Anosov endomorphism.

A natural question is whether it is possible  to find an example of  (non special)Anosov endomorphism, such that every $x \in M$ has a finite number of unstable directions. It is also interesting to understand the structure of points with infinitely many unstable directions. For transitive Anosov endomorphisms we prove the following dichotomy:
\begin{theorem}\label{teo1} Let $f: M \rightarrow M$ be a transitive Anosov endomorphism, then:
\begin{enumerate}
\item Either  $f$ is an special Anosov endomorphism,
\item Or there exists a residual subset $\mathcal{R} \subset M,$ such that for every $x \in \mathcal{R},$ $x$ has infinitely many unstable directions.

\end{enumerate}

\end{theorem}

Observe that when $M$ is the torus $\mathbb{T}^n, n \geq 2,$ all Anosov endomorphism of $\mathbb{T}^n$ are transitive, see \cite{AH}.

Analysing the unstable Lyapunov exponents of the Anosov endomorphism,  similarly with \cite{MT} we can prove:

\begin{theorem}\label{teo2} Let $A: \mathbb{T}^n \rightarrow \mathbb{T}^n, n \geq 2$ be a linear Anosov endomorphism, with $dim E^u_A = 1.$ Then there is a $C^1$ open set $\mathcal{U},$ containing $A,$ such that for every $f \in \mathcal{U}$ a $C^{1 + \alpha}, \alpha> 0,$  conservative Anosov endomorphism we have $\lambda^u_f(x) \leq \lambda^u(A),$ for $m$ almost everywhere $x \in \mathbb{T}^n,$ where $m$ is the Lebesgue measure of $\mathbb{T}^n.$
\end{theorem}

\begin{remark}
To prove the Theorem \ref{teo2}, the neighbourhood $\mathcal{U}$ is can be chosen very small, such that every $f \in \mathcal{U}$ has its lift conjugated to $A$ in $\mathbb{R}^n.$ Then by this fact, we can consider a priori that also we have $dim E^u_f = 1.$
\end{remark}

\section{General Preliminaries Results.}\label{section.preliminaries}

In this section we present some classical results on the theory of Anosov endomorphism, that will be important for our purposes for the rest of this work.

\subsection{The Limit Inverse Space.} Consider $(X,d) $ a compact metric space and $f: X \rightarrow X$ a continuous map, we define  a new compact metric space called limit inverse space for $f$ or natural extension of $f,$ being:

$$X^f := \left\{(x_n)_{n \in \mathbb{Z}} \in \prod_{i \in \mathbb{Z}} X_i|\;\; X_i = X, \;\forall \; i \in \mathbb{Z} \;\;\mbox{and}\;\; f(x_i) = x_{i+1} \forall i\in \mathbb{Z}  \right\}.$$

In this text  we denote  $X^f$ by $\widetilde{X}.$ Also we go to denote $\x$ being an element $(x_n)_{n \in \mathbb{Z}}$  of $\widetilde{X}.$ We introduce a metric $\widetilde{d}$ in $\widetilde{X}$ as following:

$$\tilde{d}(\tilde{x}, \tilde{y}) = \displaystyle\sum_{i \in \mathbb{Z}}\frac{d(x_i, y_i)}{2^{|i|}}.$$

It is easy to see  that $(\tilde{X}, \tilde{d})$ is a compact metric space.
Consider $\pi: \widetilde{X} \rightarrow X,$ the projection in the zero coordinate, that is, if $\x = (x_n)_{n \in \mathbb{Z}}, $ then $\pi(\x) = x_0.$
One can verify that $\pi$ is continuous.

\begin{definition} We denote a pre-history of $x,$  a sequence of type $\x_{-} = (\ldots, x_{-2}, x_{-1}, x_0 = x),$ such that $f(x_{-i})= x_{-i + 1}, i =1, 2, \ldots.$
\end{definition}

Denote by $X^f_{-}$ or $\widetilde{X}_{-},$ the space of the all pre-histories with $x_0 \in X.$  The space $(\widetilde{X}_{-}, \widetilde{d})$ also is compact and the distance of two pre-histories of the same point $x_0 \in X$ is $\displaystyle\sum_{i=0}^{\infty}\frac{d_M(x_{-i}, y_{-i})}{2^i}.$

In the Anosov endomorphism context, $E^u(\x)$  depends only on $\x_{-},$ and this is why many times in this work we deal only with pre histories.

\subsection{Some Nice Properties of Anosov Endomorphisms.}
The set of $C^1$ Anosov endomorphisms is open like Anosov diffeomorphisms. However, structural stability in the usual sense does not hold for Anosov endomorphisms (See the correct context for structural stability of Anosov endomorphisms in  Berger-Rovella \cite{BerRov}).
\begin{theorem}[Przytycki \cite{PRZ}, Ma\~{n}\'{e}-Pugh \cite{MP}]
Anosov endomorphisms of a manifold $M$ is an open set in the $C^1$ topology.
\end{theorem}

\begin{theorem} \label{prz} \cite{PRZ} Let $f: M \rightarrow M$ be an Anosov endomorphism, then the map

$$\tilde{x} \mapsto E^u (\tilde{x})$$
is continuous.
\end{theorem}

\begin{definition}
Let $f: M \rightarrow M$ be an Anosov Endomorphism, Denote by $\mathcal{E}^u_f(x) := \displaystyle\bigcup_{\tilde{x}: \pi(\tilde{x}) = x} E^u(\tilde{x})$,  union of all unstable direction defined on $x.$
\end{definition}
Considering the definitions \ref{defmp} and \ref{defprz} a natural question arises: What is the relation between $\mathcal{E}^u_f(x)$ and $ \bigcup_{y \in p^{-1}(x)} Dp(E^u_{\overline{f}}(y)) ?$

Observe that $\mathcal{E}^u_f(x)$ is not necessarily $\displaystyle\bigcup_{\pi(y) = x} Dp(y) \cdot ( E^u_{\widetilde{f}}(y)).$ Indeed, the latter is a countable union and the former may be uncountable  (see \cite{PRZ}.)

\begin{proposition}Let $f: M \rightarrow M$ be an Anosov Endomorphism, then

$$ \mathcal{E}^u_f(x)=\overline{\displaystyle\bigcup_{p(y) = x} Dp(y) \cdot( E^u_{\overline{f}}(y))}.$$

\end{proposition}

\begin{proof}

First of all we would like to mention that $ \mathcal{E}^u_f(x)$ is a closed subset of $u-$dimensional grassmanian of $T_xM.$ This is an immediate corollary of theorem \ref{prz}. Clearly $\displaystyle\bigcup_{\pi(y) = x} Dp(y) \cdot( E^u_{\widetilde{f}}(y)) \subset \mathcal{E}^u_f(x). $  So,  $\overline{\displaystyle\bigcup_{\pi(y) = x} Dp(y) \cdot( E^u_{\overline{f}}(y)} )\subseteq \mathcal{E}^u_f(x).$
Now for reciprocal inclusion,
let $E^u(\x)$ be an unstable direction of $x \in M.$ We want to prove that $E \in \overline{\displaystyle\bigcup_{p(y) = x} Dp(y) \cdot( E^u_{\overline{f}}(y))}.$

We claim that given any finite pre-history $(x_{-k}, \ldots , x_{-2}, x_{-1}, x=x_0 ),$ there is a finite piece of $\overline{f}-$orbit  $(y_{-k}, \cdots, \overline{f}^k(y_{-k})),$ which projects on $(x_{-k}, \ldots , x_{-2}, x_{-1}, x ),$ that is $$\pi( \overline{f}^{j}y_{-k}) = x_{-k+j}, j \in \{1, \cdots, k\}.$$

 Indeed, choose any $y_{-k} \in \overline{M},$ such that $p(y_{-k}) = x_{-k}.$ As $p\circ \overline{f} = f \circ p,$ the piece of orbit of $y_{-k}$ by $\overline{f}$ orbit projects on  $(x_{-k}, \ldots , x_{-2}, x_{-1}, x ).$

Now for each $k$ consider  $\mathcal{O}(y_{-k}),$ the full orbit by $\overline{f}$ of $y_{-k}.$ It is clear that $p(\mathcal{O}(y_{-k}))$ converges to $\tilde{x}$ in $M^f.$ Recall that
\begin{equation} \label{kk}
E^u(p(\mathcal{O}(y_{-k}))) = Dp (E^u(\overline{f}^k(y_{-k}))).
\end{equation}

  By continuity argument (theorem \cite{PRZ}) we have
  $$
  E^u(\pi(\mathcal{O}(y_{-k}))) \rightarrow E^u(\tilde{x})
  $$
  and using \ref{kk} we obtain
  $$
  Dp (E^u(\overline{f}^k(y_{-k}))) \rightarrow E^u(\tilde{x})
  $$
  which completes the proof.

\end{proof}

The next lemma is  useful for the rest of this paper.

\begin{lemma}\label{angle} Suppose that $f: M \rightarrow M$ is an Anosov endomorphism, such that  there are two different unstable directions $E^u_1$ and $E^u_2$ at $x,$ then the angle $\angle(Df^n(x)(E^u_1), Df^n(x)(E^u_2) ),$ goes to zero when $n \rightarrow + \infty.$
\end{lemma}

\begin{proof}  In fact, suppose that $dim(E^s) = k, \dim(E^u) = n .$  Suppose that $E^u_1(x) \neq E^u_2(x),$ for  $x \in M.$ Consider $\{v_1, \ldots, v_n\}$ and $\{u_1, \ldots, u_n\}$ basis for $E^u_1(x)$ and $E^u_2(x)$ respectively. Since $E^u_1 (x) \neq E^u_2 (x)$ there is $u_i,$ say $u_1,$ such that $B = \{u_1, v_1, \ldots, v_n\}$ is a linearly independent set.

 Let $E :=  < u_1, v_1, \ldots, v_n >$ with $dim(E) = n+ 1$ be the subspace generated by $B.$ Observe that $dim(E) + dim(E^s) = n + k + 1 > n+k = dim(T_xM).$ This implies that $E \cap E^s$ is non trivial. Let $0 \neq v_s \in E \cap E^s,$ we have

$$v_s =  cu_1 + v, $$

where   $c \neq 0$ and $ v \in E^u_{1} (x).$

Considering the following properties of vectors in stable and unstable bundles:
\begin{itemize}
\item $||Df^n(x) v_s|| \rightarrow 0,$
 $||Df^n(x) u_1|| \rightarrow + \infty,$
\quad \text{and} \quad $||Df^n(x) v ||\rightarrow + \infty$
\end{itemize}
 It come out  that $\angle([Df^n(x)u_1], Df^n(x)E^u_1(x)) \rightarrow 0.$ In fact the same argument shows that $\angle([Df^n(x)u_i], Df^n(x)E^u_1(x)) \rightarrow 0,$ for all $u_i$ not in $E^u_1(x).$ Thus  $$\lim_{n \rightarrow \infty}\angle(Df^n(x)(E^u_1 (x)), Df^n(x)(E^u_2(x)) )= 0.$$

\end{proof}

\section{Proof of the Theorem \ref{teo1}.}

In the course of the proof of the main result we need to analyse the number of unstable directions as a function of $x \in M$ as follows:
 Let $u: M \rightarrow  \mathbb{N} \cup \{\infty\}$ be defined as  $$u(x) := \#( \mathcal{E}_f^u(x)),$$ which assigns to each $x$ the``number" of all possible unstable directions at $T_x{M}$.

A simple and useful remark is the following:
\begin{lemma} \label{nondecreasing}
 $u(x)$ is non-decreasing along the forward orbit of $x.$
\end{lemma}
\begin{proof}
It is enough to use that $f$ is a local diffeomorphisms and $Df(x)$ is injective. However, we emphasize that it is not clear whether $u(x)$ is constant or not along the orbit. This is because, all the pre history of $x$ is included in the prehistory of $f(x).$
\end{proof}

%
%
%
%
%
%
%
%
%
%
%
%
%
%

\begin{proposition}\label{quase} Let $f: M \rightarrow M$ be a transitive Anosov endomorphism,  then either there is $x \in M$ such that $u(x) = \infty,$ or $u$ is uniformly bounded on $M,$ in fact in this case $f$ is an special Anosov endomorphism.
\end{proposition}

\begin{proof} Suppose that  $u(x) < \infty$ for all $x \in M.$    Define the sets

$$\Lambda_k =\{x \in M\, | \, u(x) \leq k \}.$$

The sets $\Lambda_k$ are closed. Indeed by continuity argument (theorem \ref{prz}) it comes out that $M \setminus \Lambda_k$ is open. Now observe that
$$ M = \displaystyle\bigcup_{k=1}^{+\infty}\Lambda_k, $$
by Baire  Category theorem,  there is $k_0 \geq 1$ such that $int(\Lambda_{k_0}) \neq \emptyset.$

 Now we claim that $$M= \Lambda_{k_0}.$$ To prove the claim, take $x$ arbitrary in $M$ with $l$ unstable directions and $V_x$ a small neighbourhood  of $x$ such that any point in $V_x$ has at least $l$ unstable direction.  Consider a point with dense orbit in $V_x$ and take an iterate of it that belongs to $\Lambda_{k_0}.$
 By lemma \ref{nondecreasing} we conclude that $l \leq k_0$ and this yields $M = \Lambda_{k_0}.$

Finally, we prove that $ M = \Lambda_1$ implying that $f$ is an special Anosov endomorphism. Suppose that, there is $x \in M$ such that $u(x) \geq 2$ and choose $E^u_1(x), E^u_2(x)$ two different unstable directions at $T_x(M).$ Let  $\alpha > 0 $ be the angle between $E^u_1(x)$ and $E^u_2(x).$

 Consider $U_x$ a small neighbourhood of $x,$ such that every $y \in U_x$ has at least two unstable directions, say $E^u_1(y)$ and $E^u_2(y),$ with $\angle(E^u_1(y), E^u_2(y)) > \displaystyle\frac{\alpha}{2}.$

 Let $x_0$ be a point with dense orbit.  Let $n_1$ be a large number satisfying
 \begin{itemize}
 \item $f^{n_1}(x_0) \in U_x, $
 \item $\angle (Df^{n_1}(x_0)\cdot E,Df^{n_1}(x_0)\cdot F ) < \displaystyle\frac{\alpha}{3} ,$ for any $E,F \in \mathcal{E}^u_f(x_0).$
 \end{itemize}
 The choice of $n_1$ is possible thanks to denseness of the forward orbit of $x_0$ and lemma \ref{angle}.
 By definition of $U_x$ the two above properties imply that  either $E^u_1(f^{n_1}(x_0))$ or $E^u_2(f^{n_1}(x_0))$ is not contained in $Df^{n_1}(x_0)\cdot\mathcal{E}^u_f(x_0).$  So, we obtain
  $$u(f^{n_1}(x_0)) > u(x_0) + 1.$$ By repeating this argument, it is possible to obtain an infinite sequence $ f^{n_k}(x_0)$ such that $$u(f^{n_{k+1}}(x_0)) \geq u(f^{n_{k}}(x_0)) + 1,$$ contradicting $M = \Lambda_{k_0}.$

\end{proof}

%
%
%
%
%
%
%
%

\subsection{Ending the Proof of Theorem \ref{teo1}}
To finalize the proof of the theorem \ref{teo1} it remains to prove that $u(x) = \infty,$ for a residual set $\mathcal{R} \subset M,$ whenever $f$ is not special Anosov endomorphism.
In fact, suppose that there is $x \in M,$ such that $u(x) = + \infty.$ Given $k > 0,$ fix exactly $k$ different unstable directions at $x,$ and $U^k_x$ a neighbourhood of $x,$ such that $u(y) \geq k,$ for every $y \in U^k_x.$ Now, since $f$ is transitive, the open set  $V^k = \displaystyle\bigcup_{i \geq 0} f^i(U^k_x) $ is dense in $M.$ Finally, consider

$$\mathcal{R} := \bigcap_{k \geq 1} U_k,$$

which is a residual set. By construction, given $x \in \mathcal{R}$ we have  $u(x) \geq k$  for every $k > 1,$ which implies $u(x) = +\infty.$ The completes the  proof of theorem \ref{teo1}.

\section{ Proof of Theorem \ref{teo2}.}

Given $f: \mathbb{T}^n \rightarrow \mathbb{T}^n$ an Anosov endomorphism, by the proposition \ref{propMP}, the lift $\overline{f}: \mathbb{R}^n \rightarrow \mathbb{R}^n$ is an Anosov diffeomorphism.

Let $f_*: \mathbb{T}^n \rightarrow \mathbb{T}^n$ be the linearisation of $f$. By linearisation of $f$ we mean the unique linear endomorphism of torus and homotopic to $f.$ By Theorem 8.1.1 in \cite{AH}, the linearisation map is hyperbolic.

Although $\mathbb{R}^n$ is not compact, since $\overline{f}$ preserves $\mathbb{Z}^n,$  the derivatives of $\overline{f}$ are periodic in fundamental compact domains of $\mathbb{T}^n.$ This periodicity allows us to prove, in the $\mathbb{R}^n$ setting, analogous results to Anosov diffeomorphisms in compact case.

\begin{lemma}\label{abscont} Let $f: \mathbb{T}^n \rightarrow \mathbb{T}^n$ be a $C^{1+\alpha}-$ Anosov endomorphism. Then for $\overline{f}: \mathbb{R}^n \rightarrow \mathbb{R}^n$ there exist $\mathcal{F}^u_{\overline{f}}$  and $\mathcal{F}^s_{\overline{f}}$ transversally absolutely continuous foliations tangent to $E^u_{\overline{f}}$ and $E^s_{\overline{f}}$ respectively.
\end{lemma}

\begin{proof}
Similar to the compact case, \cite{HPS}.
\end{proof}

\begin{definition} \label{quasi isometric}
A foliation $W$ of $\mathbb{R}^n$  is quasi-isometric if there exist positive constants $Q$ and $b$
such that for all $x, y$ in a common leaf of W we have
$$d_W(x, y) \leq Q^{-1} || x - y|| + b.$$
Here $d_W$ denotes the riemannian metric on $W$ and $\|x-y\|$ is the euclidean distance.
\end{definition}

\begin{remark}\label{remarkquasi} Observe that if $||x - y||$ is large enough, we can  consider $b = 0, $ in the above definition.
\end{remark}

\begin{lemma}\label{quasi_iso_fol}
Let $A$ be as theorem \ref{teo2}. If   $f$ is an Anosov endomorphism  $C^1-$ sufficiently close to $A,$ then $\mathcal{F}^{s,u }_{\overline{f}}$ are quasi isometric foliations.
\end{lemma}

The proof of this lemma follows directly from a proposition due to Brin, \cite{Br}.

\begin{proposition}\label{brin} Let $W$ be a $k-$dimensional foliation on the $\mathbb{R}^m.$ Suppose that there is a $(m-k)-$dimensional plane $\Delta$ such that $T_x W(x) \cap \Delta  =\{0\},$ such that $\angle (T_x W(x) , \Delta) \geq \beta > 0,$ for every $x \in \mathbb{R}^m.$ Then $W$ is quasi isometric.
\end{proposition}

\begin{proof}[Proof of the lemma]
 Consider $U $  a $C^1-$open set containing $A,$ such that for every $f \in U,$  $\overline{f}$ and $\overline{A}$ are $C^1$ close in the universal cover $\mathbb{R}^n.$

 The $C^1$ neighborhood $U,$ is taken such that

\begin{equation}\label{angulou}
|\angle (E^u_{\overline{f}}(x), E^u_A ) | < \alpha,
\end{equation}

\begin{equation}\label{angulos}
|\angle (E^s_{\overline{f}}(x), E^s_A ) | < \alpha,
\end{equation}

 for any $x \in \mathbb{R}^n$ where $\alpha$ is an small number less than 1/2$\angle(E^u_A, E^s_A).$
For the foliation $\mathcal{F}^u_{\overline{f}}$ take $\Delta := E^s_A,$ and, for the foliation $\mathcal{F}^s_{\overline{f}},$  $\Delta := E^u_A.$
  Applying the proposition \ref{brin} completes the proof.

 \end{proof}


\begin{corollary}[Nice Properties]\label{nice}
 For any Anosov endomorphism $f: \mathbb{T}^n \rightarrow \mathbb{T}^n$ close to its linearisation $A$, the following properties hold in the universal covering:

\begin{enumerate}

\item For each $k \in \mathbb{N}$ and $C > 1$ there is $M$ such that,
$$||x -  y|| > M \Rightarrow \frac{1}{C} \leq \displaystyle\frac{||\overline{f}^kx - \overline{f}^k y|| }{||A^kx - A^ky||} \leq C.$$

\item  $ \displaystyle\lim_{||y - x || \rightarrow +\infty} \frac{y-x}{||y - x ||} = E_A^{\sigma}, \;\;  y \in \mathcal{F}^{\sigma}_{\bar{f}} (x), \sigma  \in \{s,  u\},$
uniformly.
\end{enumerate}
\end{corollary}

\begin{proof}
The proof is in the lines of \cite{H} and we repeat for completeness. Let $K$ be a fundamental domain of $\mathbb{T}^d$ in $\mathbb{R}^d, d \geq 2.$ Restricted to $K $ we have

$$||\overline{f}^k -  A^k|| < +\infty, $$

For $\overline{x} \in \mathbb{R}^d,$ there are $x \in K$ and $\overrightarrow{n} \in \mathbb{Z}^d,$ such that $\overline{x}  = x + \overrightarrow{n}, $ since $f_{\ast} = A,$ we obtain:

$$ ||\overline{f}^k(\overline{x}) - A^k(\overline{x})|| = ||\overline{f}^k(x + \overrightarrow{n}) - A^k(x +\overrightarrow{ n})|| = ||\overline{f}^k(x) +A^k \overrightarrow{n} - A^kx - A^k\overrightarrow{n}  || < +\infty. $$

Now, for every $x, y \in \mathbb{R}^d,$

$$||\overline{f}^k x - \overline{f}^ky|| \leq ||A^k x - A^k y|| + 2||\overline{f}^k - A^k||_0$$

$$||A^k x - A^ky|| \leq ||\overline{f}^k x - \overline{f}^k y|| + 2||\overline{f}^k - A^k||_0,$$

where,

 $$||\overline{f}^k - A^k||_0 = \max_{x \in K}\{||\overline{f}^k(x) - A^k(x)||\}.$$

Since $A$ is non-singular,  if $||x -  y|| \rightarrow + \infty,$ then $||A^kx - A^k y|| \rightarrow + \infty.$

So dividing both expressions by $||A^kx - A^k y|| $ and doing $||x -  y|| \rightarrow + \infty$ we obtain the proof of the first item.

For the second item, we just consider the case of $E^s_A,$ for $E^ u$ just take $A^{-1}$ and $(\overline{f})^{-1}$  and same proof holds.

Let $|\theta^s| = \max\{\,|\theta| \;| \; \theta\; \mbox{is eigenvalue of $A$ and}\; 0 < |\theta| < 1 \}.$  Fix a small $\varepsilon > 0$ and consider $\delta > 0,$ such that $0 < (1+ 2\delta)|\theta^s| < 1.$ If $f$ is sufficiently $C^1-$close to $A,$ then $\overline{f}$  is an Anosov diffeomorphism  on $\mathbb{R}^d$ with contracting constant less than $(1+ \delta)|\theta^s|.$

Using the hyperbolic splitting, there is $k_0 \in \mathbb{N},$ such that if $v \in \mathbb{R}^d,$  $k > k_0$ and

$$||A^k v || < (1 + 2\delta)^k |\theta^s|^k ||v||,$$

then

$$||\pi^u_A(v)|| < \varepsilon ||\pi^s_A(v)||.$$

Consider $k > k_0$ and $M$ sufficiently large, satisfying  the first item  with $C = 2$ and in accordance with remark \ref{remarkquasi}.

Take $y \in \mathcal{F}^s_{\overline{f}}(x)$ and $||x - y|| > M.$ Let $d^s$ to denote the riemannian distance on stable leaves of $\mathcal{F}^s_{\overline{f}},$ using quasi isometry property of the foliation $\mathcal{F}^s_{\overline{f}},$ we get

$$d^s(\overline{f}^k x, \overline{f}^k y) <  ((1 + \delta)|\theta^s|)^k d^s(x,y) \Rightarrow $$
$$ ||\overline{f}^k x - \overline{f}^k y|| <  ((1 + \delta)|\theta^s|)^k (Q^ {-1}|| x - y||) \Rightarrow $$
$$ ||A^k x - A^k y|| < 2 ((1 + \delta)|\theta^s|)^k (Q^ {-1}|| x - y||).$$

Finally,  for large $k$ we have:

$$2Q^{-1} ((1 + \delta)|\theta^s|)^k  \leq ((1 + 2\delta) |\theta^s|)^k,$$

So,

$$||\pi^u_A(x - y)|| < \varepsilon ||\pi^s_A(x- y)||.$$

\end{proof}

\begin{lemma} \cite{MT} \label{linalg}
 Let $f : \mathbb{T}^d \rightarrow \mathbb{T}^d$ be an Anosov endomorphism close to  $A: \mathbb{T}^d \rightarrow \mathbb{T}^d,$ such that $dim E^u_A = 1.$ Then   for all $n \in \mathbb{N}$ and $\varepsilon > 0$ there exists $M$ such that for $x, y$ with $y \in \mathcal{F}^u_{\overline{f}}(x)$ and $||x -  y||> M$ then
 $$
   (1 - \varepsilon)e^{n\lambda^{u}_A } ||y -x|| \leq \|A^n(x) - A^n(y)\| \leq (1 + \varepsilon)e^{n\lambda^{u}_A } ||y -x||
 $$
where $\lambda^{u}$ is the Lyapunov exponent of $A$ corresponding to $E^{u}_A.$
 \end{lemma}

\begin{proof} Denote by $E^{u}_A$ the eigenspace corresponding to $\lambda^{u}_A$ and $|\mu| = e^{\lambda^{u}_A},$ where $\mu$ is the eigenvalue of $A$ in the $E^{u}_A$ direction.

Let  $N \in \mathbb{N}$ and choose $x, y \in \mathcal{F}^{u}_f(x),$ such that $|| x - y || > M.$ By corollary \ref{nice},we have

$$ \frac{x - y}{|| x - y||} = v + e_M,$$

where the vector $v = v_{E^{u}_A}$ is a unitary eigenvector of  $A,$ in the  $E^{u}_A$ direction and  $e_M$ is a correction vector that converges to zero uniformly as  $M$ goes to  infinity. We have

$$A^N \left( \frac{x - y}{|| x - y||} \right) = \mu^N v + A^N e_M = \mu^N \left(\frac{x - y}{|| x - y||} \right) -\mu^N e_M  + A^N e_M  $$

It implies that

\begin{align*} || x - y || (|\mu|^N - |\mu|^N ||e_M|| - ||A||^N || e_M||) \leq   || A^N (x - y)|| \\ \leq || x - y || (|\mu|^N + |\mu|^N ||e_M|| + ||A||^N || e_M||).
\end{align*}

Since $N$ is fixed, we can choose  $M > 0,$ such that

$$ |\mu|^N ||e_M|| + ||A||^N || e_M|| \leq \varepsilon |\mu|^N.$$

and the lemma is proved. \end{proof}

\begin{remark}
By Multiplicative Ergodic Theorem for endomorphisms (\cite{QXZ}) the unstable Lyapunov exponent for a typical point is independent of unstable direction. We denote by $\lambda^u(x) = \lambda^u(\tilde{x})$ the unique unstable Lyapunov exponent of $x$ in our context where $\dim (E^u) =1.$
\end{remark}




\begin{theorem}[Theorem \ref{teo2}] Let $A: \mathbb{T}^n \rightarrow \mathbb{T}^n, n \geq 2$ be a conservative linear Anosov endomorphism, with $dim E^u_A = 1.$ Then there is a $C^1$ open set $\mathcal{U},$ containing $A,$ such that for every $f \in \mathcal{U}$ a $C^{1+\alpha}, \alpha > 0$  conservative Anosov endomorphism we have $\lambda^u_f(x) \leq \lambda^u_A,$ for $m$ almost everywhere $x \in \mathbb{T}^n,$ where $m$ is the Lebesgue measure of $\mathbb{T}^n.$
\end{theorem}

\begin{proof}

Suppose by contradiction that there is a  positive set $Z \in \mathbb{T}^n,$ such that, for every $x \in Z$ we have  $\lambda^u_{\overline{f}}(x) > (1 + 5 \varepsilon) \lambda^u_A,$ for a small $\varepsilon > 0.$
Since $\overline{f}$ is  $C^{1+\alpha},$ the unstable foliation $\mathcal{F}_{\overline{f}}^u$ is upper absolutely continuous. So, there is a positive Lebesgue measure set $B \in \mathbb{R}^n$ such that for every point $x \in B$ we have

\begin{equation}
 m^u_x(\mathcal{F}^u_{\overline{f}}(x) \cap Z) > 0
\label{1}
\end{equation}
where $m^u_x$ is the Lebesgue measure of the leaf $\mathcal{F}^u_{\overline{f}}(x)$.
Choose a $p \in B$ satisfying (\ref{1}) and consider an interval $[x,y]_u \subset \mathcal{F}^u_{\overline{f}}(p) $ satisfying
 $m^u_p([x,y]_u \cap Z) > 0$ such that the length of $[x,y]_u$ is bigger than $M$ as required in the lemma \ref{linalg} and corollary \ref{nice}. We can choose $M$ such that

$$||Ax - Ay|| < (1 + \varepsilon)e^{\lambda^u_A } ||y -x|| $$

and

$$\frac{|| \overline{f}(x) - \overline{f}(y)|| }{ ||Ax - Ay||} < 1 + \varepsilon. $$

 whenever
 $d^u(x, y) \geq M,$ where $d^u$ denotes the riemannian distance in unstable leaves. The above equation implies that $$ || \overline{f}(x) -\overline{f}(y)|| < (1+ \varepsilon)^2 e^{\lambda^u_A} || y - x||.$$

Inductively, we assume that for  $n \geq 1$ we have

\begin{equation}
|| \overline{f}^n(x) - \overline{f}^n(y)||< (1+\varepsilon)^{2n} e^{n \lambda^u_A }|| y - x||. \label{induction}
\end{equation}

 Since $f$ expands uniformly  on the $u-$direction we have $d^u(\overline{f}^n(x), \overline{f}^n(y)) > M,$ consequently

\begin{eqnarray*}
||\overline{ f}(\overline{f}^nx) - \overline{f}(\overline{f}^ny)|| &<& (1+\varepsilon)|| A(\overline{f}^nx) - A(\overline{f}^ny)|| \\  &<& (1 + \varepsilon)^2 e^{\lambda^u_A} || \overline{f}^nx - \overline{f}^n y||\\ &<&
 (1+\varepsilon)^{2(n+1)} e^{(n+1)\lambda^u_A}.
\end{eqnarray*}

For each $n > 0,$ let $A_n \subset Z$ be the following set

$$A_n = \{ x \in Z \colon\;\; \|D\overline{f}^k(x)|E^u_{\overline{f}}(x) \| > (1+2\varepsilon)^{2k} e^{k\lambda^u_A} \;\; \mbox{for any} \;\; k \geq n\}. $$
We have $m(Z) > 0$ and $Z_n := (A_n \cap Z) \uparrow Z,$ as $(1 + 5 \varepsilon) > (1 + 2\varepsilon)^2,$ for small $\varepsilon > 0.$

Define the number $\alpha_0 > 0$ such that:

$$\displaystyle\frac{m^u_p([x,y]_u \cap Z)}{m^u_p([x,y]_u)} = 2 \alpha_0.$$

Since $Z_n \cap [x,y]_u \uparrow Z \cap [x,y]_u, $ there is $n_0 \in \mathbb{N} ,$ such that $n \geq n_0,$ then

 $$ m^u_p ([x,y]_u \cap Z_n) = \alpha_n \cdot m^u_p([x,y]_u),$$

for $\alpha_n > \alpha_0.$

Thus, for $n \geq n_0$ we have:

\begin{eqnarray}
||\overline{f}^nx - \overline{f}^ny || &>&  Q \displaystyle\int_{[x,y]_u \cap Z_n} ||Df^n(z)|| dm^u_p(z) >  \\ &>&
 Q (1+ 2\varepsilon)^{2n}
e^{n \lambda_A^u } m^u_p ([x,y]_u \cap Z_n) \\ &>& \alpha_0 Q^2 (1 + 2\varepsilon)^{2n} e^{n\lambda^u_A} \|x-y\|. \label{conclusion}
\end{eqnarray}
The inequalities $(\ref{induction})$ and $(\ref{conclusion})$ give a contradiction.

\end{proof}

\section{Appendix: Ergodicity Of Anosov Endomorphisms }

In this section we establish the ergodicity of $C^{1+\alpha}$ conservative Anosov Endomorphism. Before this, we remember a classical  result, see \cite{QXZ}.

\begin{proposition} Let $T:  X \rightarrow X $ a continuous map of a compact metric space. For any $T-$invariant Borel probability
measure $\mu$ on $X,$ there exists a unique $\tilde{T}-$ invariant Borel probability measure
$\tilde{\mu}$ on $X^T$ such that $\pi_{\ast}\tilde{\mu} = \mu.$ In particular for a measurable set $A \subset X,$ we have $\mu(A) = \tilde{\mu}(\pi^{-1}(A)). $ Moreover $\mu$ is ergodic if, and only if, $\tilde{\mu}$ is ergodic. Here $\pi: X^T \rightarrow X$ is the projection in the zero coordinate.
\end{proposition}

To prove the ergodicity, we use the theory of SRB-measures of endomorphisms. We suggest the reader to see \cite{QZ}, for further details on SRB theory of  endomorphisms. See also \cite{BT} for the number of ergodic SRB measures for surface endomorphisms in terms of homoclinic equivalence classes.

\begin{theorem}
 Let $f: M \rightarrow M$ be a  $C^{1+\alpha}-$ conservative Anosov Endomorphism. Then $(f,m)$ is ergodic, where $m$ is a volume form on $M.$
\end{theorem}

The above theorem seems to be folklore in the ergodic theory of hyperbolic dynamics. However,  we did not find any written proof.
Initially, observe that if $f$ is an Anosov endomorphism, then $f$ is Axiom. $A$ . Since we are supposing that $f$ is $m-$preserving, then $\Omega(f) = M$ and $f$ is transitive (see \cite{MP} or \cite{PRZ}).

\begin{proof} By Pesin formula for endomorphisms, we know that

$$h_m(f) = \int_M \sum_{i} \lambda_i^{+}(x)m_i(x) dm,$$

where $m_i(x)$ is the algebraic multiplicity of $\lambda_i(x).$ In fact we want to prove that the ergodic decomposition of $m$ is trivial. If $\mu$ is not ergodic, by ergodic decomposition theorem we can suppose that $m$ admits at least two ergodic components $\mu_1$ and $\mu_2,$  such that

\begin{equation}\label{SRB}
h_{\mu_k}(f) = \int_M \sum_{i} \lambda_i^{+}(x)m_i(x) d\mu_k, k=1,2.
\end{equation}
Denote by $B_i = B(\mu_i), i=1,2$ the basins of measures $\mu_1$ and $\mu_2$ respectively,

$$ B_i = \left\{ x \in M| \; \frac{1}{n} \sum_{j = 1}^{n-1} \varphi(f^j(x)) \rightarrow \int_M \varphi d\mu_i\right\}, i =1,2,$$

for every $\varphi \in C^0(M),$ we have $\mu_i(B_i) = 1.$

By the SRB characterization given in \cite{QZ}, the measures $\mu_1$ and $\mu_2$ are SRB, since $\mu_1$ and $\mu_2$ satisfies the formula $(\ref{SRB}).$

Moreover, using the SRB theory developed in \cite{QZ}, the measures $\tilde{\mu}_1, \tilde{\mu}_2,$ are SRB measure, and for $\tilde{m},$ a.e. $\tilde{x} \in M^f,$ we have

$$\pi(\tilde{\mu}_i^{\eta(\tilde{x})}(\tilde{x})) << m^u_{\tilde{x}}, i =1,2,$$

where $\eta(\tilde{x})$ is the atom of a subordinated partition for $\tilde{m}$ and $m^u_{\tilde{x}},$ is the Lebesgue volume defined on $\mathcal{F}^u_f({\tilde{x}}).$

Since $\tilde{\mu}_i(B(\tilde{\mu}_i)) = 1, i =1,2$ we have that $\mu_i(\pi(B(\tilde{\mu}_i))= 1.$ By absolute continuity of conditional measures with respect to Lebesgue measure (in fact equivalence), there exist  $\tilde{x}_1, \tilde{x}_2,$ such that the set

$$F^u_i := B_i \cap \pi(B(\tilde{\mu}_i))\cap \pi(\eta(\tilde{x}_i)), i = 1, 2 $$

has full Lebesgue measure in $\pi(\eta(\tilde{x}_i)), i =1,2.$

 Now we saturate $F^u_i $ by leaves of $\mathcal{F}^s_{f}.$ That is we take $D_i: = \displaystyle\bigcup_{z \in F^u_i } \mathcal{F}^s_f(z),$ as union of stable leaves through points of $F^u_i.$
As $\mathcal{F}^s_f$ is a  continuous foliation,  these saturations contain open sets modulus zero (w.r.t. $m$).

Now, if $z_i \in D_i,  $ then $\mathcal{F}^s_f(z_i )$ intersects $F^u_i$ in a point $y_i.$ Since $y_i, z_i$ are in the same stable leaf

$$ \lim_{n \rightarrow \infty} \frac{1}{n} \sum_{j = 1}^n \varphi(f^j(z_i)) =  \lim_{n \rightarrow \infty}  \frac{1}{n} \sum_{j = 1}^n \varphi(f^j(y_i)) \rightarrow \int_M \varphi d\mu_i , i =1,2 $$

thus $D_i \subset B_i, i =1,2.$

Now, since $f$ is transitive, there is $N$ such that $f^N(D_1) \cap D_2 \neq \emptyset$ in an open set (modulo a null $m$ set), since $f^N(B_1) \subset B_1,$ in particular $B_1 \cap B_2 \neq \varnothing,$ so $\mu_1 = \mu_2,$ absurd.

\end{proof}

\end{document}